\documentclass[12pt]{amsart}

\usepackage{amssymb}
\usepackage{bm}
\usepackage{graphicx}
\usepackage{psfrag}
\usepackage{color}
\usepackage{url, enumerate}
\usepackage{algpseudocode}

\newcommand{\excise}[1]{}

\newtheorem{thm}{Theorem}[section]
\newtheorem{lemma}[thm]{Lemma}

\newtheorem{cor}[thm]{Corollary}
\newtheorem{prop}[thm]{Proposition}

\theoremstyle{definition}
\newtheorem{alg}[thm]{Algorithm}
\newtheorem{example}[thm]{Example}
\newtheorem{remark}[thm]{Remark}
\newtheorem{defn}[thm]{Definition}

\newtheorem{prob}{Problem}

\newtheorem{conj}[prob]{Conjecture}

\numberwithin{equation}{section}

\newcommand{\ring}[1]{\ensuremath{\mathbb{#1}}}

\renewcommand\>{\rangle}
\newcommand\<{\langle}

\newcommand\NN{\ring{N}}

\newcommand\QQ{\ring{Q}}

\newcommand\ZZ{\ring{Z}}

\DeclareMathOperator\bul{bul} 

\begin{document}

\mbox{}
\title{How Do You Measure Primality?}
\author{Christopher O'Neill}
\address{Mathematics Department\\Duke University\\Durham, NC 27708}
\email{musicman@math.duke.edu}
\author{Roberto Pelayo}
\address{Mathematics Department\\University of Hawai`i at Hilo\\Hilo, HI 96720}
\email{robertop@hawaii.edu}

\date{\today}

\begin{abstract}
\hspace{-2.05032pt}
In commutative monoids, the $\omega$-value measures how far an element is from being prime.  This invariant, which is important in understanding the factorization theory of monoids, has been the focus of much recent study.  This paper provides detailed examples and an overview of known results on $\omega$-primality, including several recent and surprising contributions in the setting of numerical monoids.  As many questions related to $\omega$-primality remain, we provide a list of open problems accessible to advanced undergraduate students and beginning graduate students.

\end{abstract}

\maketitle

\section{Introduction}\label{s:intro}

The factorization theory of monoids has its genesis in understanding the decomposition of the positive natural numbers $\NN^*$ into prime numbers.  In this multiplicative monoid, the Fundamental Theorem of Arithmetic guarantees that any natural number greater than 1 can be written uniquely (up to re-ordering) as a product of prime numbers, and much of the multiplicative structure of $\NN^*$ is determined by this decomposition into primes.  

Many of the nice factorization properties enjoyed by $\NN^*$ very quickly break down when considering other monoids, as they often contain too few primes to decompose every element.  The study of factorizations in general monoids thus focuses on decomposing elements into \emph{irreducibles}, which satisfy a weaker condition than primes, and these decompositions are often non-unique.  The concept of a prime element still remains important in understanding the structure of these monoids, since prime elements are irreducibles that appear in every factorization of the elements they divide.  

In~\cite{origin}, Geroldinger develops the notion of $\omega$-primality (Definition~\ref{d:omega}), which measures how far a non-unit element is from being prime.  This $\omega$-function assigns to every non-unit element $n$ in a monoid,  a value $\omega(n) \in \NN^* \cup \{\infty\}$ with the property that $\omega(n) = 1$ if and only if $n$ is prime in the monoid.  Thus, understanding the $\omega$-function gives tremendous insight into not only which elements are prime, but also how ``far from prime'' an element is.  

In any cancellative, commutative monoid, every prime element is irreducible, but not every irreducible element is necessarily prime.  In fact, the existence of non-prime irreducible elements coincides with the existence of non-unique factorizations.  As such, much of the initial work in computing $\omega$-primality focused on the irreducible elements~\cite{origin, tame}.  More recently, though, computing $\omega$-values for arbitrary elements has produced interesting algorithms, closed formulae, and asymptotic results~\cite{andalg, quasi, compasymp}.

In certain settings, the definition of the $\omega$-function becomes rather tractable, and its study can be facilitated using available algorithms in established computer packages.  This makes many research questions on $\omega$-primality accessible to both strong undergraduate students and professional researchers alike.  

This paper focuses on providing a review of known results on $\omega$-primality, numerous examples in specific monoids, and a range of open problems.  As these problems are introduced, every effort is made to include relevant background, and all are accessible to advanced undergraduates, thereby serving as potential research problems for REU students or graduate students.  

We begin in Section~\ref{s:general} by providing two equivalent definitions of the $\omega$-function (Definition~\ref{d:omega} and Proposition~\ref{p:bullet}), along with some of its basic properties (Proposition~\ref{p:basicprops}).  Several computational examples of $\omega$-values are also worked out in detail for different monoids, and existing results for block monoids (and more generally for Krull monoids) are also discussed.  We turn our attention in Section~\ref{s:numerical} to a particular class of monoids, known as \emph{numerical monoids} (Definition~\ref{d:numerical_monoid}).  We review several results on the $\omega$-function for numerical monoids, including an algorithm suitable for computation (Algorithm~\ref{a:numerical}), a closed form for the $\omega$-function in certain special cases (Theorem~\ref{t:2genomega}), and a recent periodicity result (Theorem~\ref{t:quasi}) that gives an asymptotic description of the $\omega$-function for any numerical monoid (Corollary~\ref{c:asymptotics}).  Once again, we work out several examples in detail.  

Even with the numerous results in Sections~\ref{s:general} and~\ref{s:numerical}, many open questions remain.  We conclude this paper with several open problems related to $\omega$-primality in Section~\ref{s:open}.  This includes questions related to various aspects of numerical monoids (Sections~\ref{s:newnumerical}-\ref{s:newalgs}), as well as several other classes of monoids that are known to have interesting factorization structures, but for which little is known about $\omega$-primality (Section~\ref{s:newmonoids}).

\section{$\omega$-primality and its properties}\label{s:general}

In this section, we introduce the $\omega$-primality function and some of its basic properties, focusing on several computational examples in specific monoids.  Unless otherwise stated, all monoids are assumed to be commutative, cancellative, and atomic (i.e., non-unit elements can be written as a product of finitely many irreducibles).  In what follows, let $\NN$ denote the set of non-negative integers, and $\NN^*$ the set of positive integers.  

\subsection{The $\omega$-primality function}

Recall that a non-unit element in a monoid is \emph{irreducible} if it cannot be written as a product of non-units.  The set $\NN^*$ under multiplication forms a monoid in which every decomposition of $n \in \NN^*$ into irreducible elements is unique up to re-ordering.  Example~\ref{e:hilbert} provides an example of a submonoid of $\NN^*$ that fails to inherit this property.
  
\begin{example}\label{e:hilbert}
Let $H = \{x \in \NN^* : x \equiv 1 \bmod 4\}$ denote the set of natural numbers congruent to 1 mod 4.  Since $1^2 \equiv 1 \bmod 4$, the monoid $H$ forms a multiplicative submonoid of the positive integers called the \emph{Hilbert Monoid}, whose algebraic structure has been studied extensively~\cite{delta_acm, elas_local, acm_arithmetic, acm_elas}.  

The factorization structure of $H$ is heavily determined by prime factorization in $\NN$, since multiplication in $H$ 
coincides with multiplication in $\NN^*$.  Any positive prime $p \in \ZZ$ with $p \equiv 1 \bmod 4$ is irreducible in $H$, and therefore
cannot be written as a product of two non-unit elements of $H$.  The converse of this, however, does not always hold.  For instance, $9 \in H$ can be decomposed uniquely in $\NN^*$ as $9 = 3 \cdot 3$; since $3 \not \in H$, there is no factorization for $9$ in $H$, and it is therefore irreducible.  In fact, an element $x \in H$ 
is irreducible if and only if either $x$ is prime in $\NN^*$ or $x = p_1p_2$ for prime integers $p_1, p_2 \equiv 3 \bmod 4$ in $\NN^*$.  

Elements in $H$ can always be written as a product of irreducible elements, but these factorizations need not be unique.  
For example, consider $441 = 3^27^2 \in H$.  Since $3, 7 \equiv 3 \bmod 4$, neither lie in $H$, but $3^2$, $7^2$ and $3 \cdot 7$ all lie in $H$ 
and are irreducible by the argument above.  Thus, we can write $441 = (3^2)(7^2) = (3 \cdot 7)(3 \cdot 7)$ as 
two distinct factorizations into irreducible elements.  
\end{example}

Monoids, as algebraic objects, have a single binary operation.  Although the set of natural numbers $\NN$ has two binary operations, 
the monoid in Example~\ref{e:hilbert} focuses only on the multiplicative operation.  It is natural to consider irreducibles and 
factorizations in additive submonoids of $\NN$ as well, motivating the following example.

\begin{example}\label{e:mcnugget}
Let $M = \{6a + 9b + 20c : a, b, c \in \NN\} \subset \NN$ denote the additive submonoid of $\NN$ 
generated by 6, 9 and 20.  This monoid is referred to as the \emph{McNugget Monoid} since McDonald's 
at one time sold Chicken McNuggets in packs of 6, 9 and 20.  In particular, an integer $n$ lies in $M$ 
if it is possible to purchase $n$ Chicken McNuggets using only these pack sizes.  

Since $M$ is an additive monoid, divisibility in $M$ does not coincide with the standard notion of 
divisibility in $\NN$.  In particular, 6 divides 15 in $M$ since $6 + 9 = 15$ and $9$ lies in $M$.  
The elements $6$, $9$ and $20$ are irreducible in $M$, since each cannot be written as a linear combination 
of the other two, and every other non-zero element in $M$ is reducible by definition.  Different factorizations of 
an element $n \in M$ can be viewed as the different ways of purchasing $n$ Chicken McNuggets using packs of 
6, 9 and 20.  For instance, 18 can be factored as $9 + 9$ and $6 + 6 + 6$ in $M$, and one can purchase 18 
Chicken McNuggets by purchasing either 2 packs of 9 McNuggets or 3 packs of 6.  
\end{example}

Nonunique factorizations in monoids are often studied using several invariants, including elasticity, delta sets, and catenary degree, which measure how far an element is from having a unique factorization (see~\cite{nonuniq} for details).  This paper focuses on the $\omega$-primality invariant, which has received much attention in the recent literature.

\begin{defn}\label{d:omega}
Fix an element $x$ in a commutative, cancellative, atomic monoid $M$.  
The \emph{$\omega$-primality function} assigns to each element $x \in M$ the value $\omega_M(x) = m$ if $m$ 
is the smallest positive integer with the property that whenever $x \mid \prod_{i=1}^r q_i$ for $r > m$, 
there exists a subset $T \subset \{1, \ldots, r\}$ with $|T| \le m$ such that $x \mid \prod_{i \in T} q_i$.  
If no such $m$ exists, define $\omega_M(x) = \infty$.  When $M$ is clear from context, we simply write $\omega(x)$.  
\end{defn}

\begin{remark}\label{r:farfromprime}
Recall that an element $x$ in a monoid $M$ is \emph{prime} if it is primitive with respect to divisibility; that is, $x$ is prime if  
$x \mid q_1 \cdots q_r$ implies $x \mid q_i$ for some $i$.  Clearly, Definition~\ref{d:omega} was constructed so that  $\omega(x) = 1$ if and only if $x$ is prime.  Moreover, the larger the value of $\omega(x)$, the further $x$ is from being prime.  For instance, suppose that $\omega(x) = 2$ and 
$x \mid q_1 \cdots q_r$.  While we cannot guarantee that $x$ divides some $q_i$, we can guarantee it divides some 
product $q_iq_j$.  
\end{remark}

\begin{example}\label{e:natural}
In the multiplicative monoid $\NN^*$, every positive integer can be written uniquely as a product of prime integers in $\NN^*$.  
If a prime $p \in \NN^*$ divides a product $a_1 \cdots a_r$ of positive integers, then $p$ occurs in the prime 
factorization of some $a_i$, so $p \mid a_i$.  This means $\omega(p) = 1$; see Remark~\ref{r:farfromprime}.  
More generally, fix a positive integer $n = p_1 \cdots p_\ell \in \NN^*$ for primes $p_1, \ldots, p_\ell \in \NN^*$, 
and suppose $n$ divides a product $a_1 \cdots a_r$ of positive integers.  This means the product $p_1 \cdots p_\ell$ 
occurs within the expression $a_1 \cdots a_r$ upon replacing each $a_i$ with its prime factorization.  At worst, 
each prime factor of $n$ occurs in the prime factorization of a different $a_j$, so this bounds $\omega(n) \le \ell$.  
Since $n \mid p_1 \cdots p_\ell$, we have $\omega(n) = \ell$.  In particular, the $\omega$-value of any positive integer 
$n \in \NN^*$ equals the length of the prime factorization of $n$.  
\end{example}

Example~\ref{e:natural} demonstrates that in the monoid $\NN^*$, $\omega$-primality measures the length of an element's unique factorization.   In general, every irreducible element  in a unique factorization domain is prime, and an argument similar to the above shows that the $\omega$-value of a non-unit element in these monoids is precisely its factorization length.

\begin{example}\label{e:hilbertomega}
Let $H$ denote the Hilbert monoid defined in Example~\ref{e:hilbert}.  Any irreducible $u \in M$ not prime in $\NN$ is a product 
of two primes in $\NN$, both of which are congruent to $3 \bmod 4$.  Let $x = 1225 = 5^2 7^2$.  We claim $\omega_H(x) = 4$.  Suppose 
$x \mid u_1 \cdots u_r$ for irreducible elements $u_1, \ldots, u_r \in H$.  Both 5 and 7 must occur twice in the prime factorization of $u_1 \cdots u_r$ 
over $\NN$ since $\NN$ has unique factorizations.  This means some subset $T \subset \{1, \ldots, r\}$ with $|T| \le 4$ yields 
$x \mid \prod_{i \in T} u_i$, so $\omega_H(x) \le 4$.  Moreover,  $x \mid 5 \cdot 5 \cdot (7 \cdot 11) \cdot (7 \cdot 11)$, 
where the elements $5$ and $7 \cdot 11$ are irreducible in $H$.  Since no irreducible element can be omitted, we have $\omega_H(x) = 4$.  
\end{example}

\begin{example}\label{e:mcnuggetomega}
Let $M \subset \NN$ denote the McNugget Monoid from Example~\ref{e:mcnugget}, with irreducibles $6$, $9$ and $20$.  We claim $\omega_M(6) = 3$.  
Fix a factorization $u_1 + \cdots + u_r$ with $r > 3$ and each $u_i \in \{6, 9, 20\}$, and suppose $6$ divides $u_1 + \cdots + u_r$ in $M$
(that is,  $u_1 + \cdots + u_r - 6 \in M$).  Consider the following cases.  
\begin{itemize}
\item[$\cdot$] 
If any $u_i = 6$, then omitting the remaining elements from the factorization $u_1 + \cdots + u_r$ leaves $u_i - 6 = 0 \in M$.  

\item[$\cdot$] 
If any two elements $u_i$ and $u_j$ are both $9$, then dropping the remaining elements leaves $u_i + u_j - 6 = 12 \in M$.  

\item[$\cdot$]
If three of the elements $u_i$, $u_j$ and $u_k$ equal 20, then omitting all remaining summands leaves $u_i + u_j + u_k - 6 = 54 \in M$.  

\end{itemize}
Since $9 + 20 + 20 - 6 \notin M$, at least one of the above conditions must hold.  Thus, we can omit summands from any factorization 
$u_1 + \cdots + u_r$ that $6$ divides until either a single $6$, two $9$'s or three $20$'s remain, and $6$ will still divide the result.  
This leaves at most 3 summands, so $\omega_M(6) = 3$.  
\end{example}

\subsection{$\omega$-primality and bullets}

Notice that in order to compute $\omega(6)$ in Example~\ref{e:mcnuggetomega}, it suffices to find all collections of irreducibles with the following property:  the product of this collection of irreducibles is divisible by $6$ (in the additive sense), and if any irreducible is omitted from the product, it is no longer divisible by $6$.   Such factorizations are called bullets (Definition~\ref{d:bullet}), 
and the $\omega$-function can be equivalently defined in terms of bullets (Proposition~\ref{p:bullet}).  This characterization will play 
a crucial role in Algorithm~\ref{a:numerical}, used for computing $\omega$-values in numerical monoids.

\begin{defn}\label{d:bullet}
Fix an element $x$ in a commutative, cancellative, atomic monoid $M$.  
A \emph{bullet} for $x$ is a product $u_1 \cdots u_r$ of irreducible 
elements $u_1, \ldots, u_r \in M$ such that 
\begin{enumerate}
\item $x$ divides the product $u_1 \cdots u_r$, and 
\item for each $i \le r$, $x$ does not divide $u_1 \cdots u_r / u_i$.  
\end{enumerate}
The set of bullets of $x$ is denoted $\bul(x)$.  
\end{defn}

\begin{remark}
For a more general (not necessarily cancellative) commutative monoid $M$, expressions of the form $ab/a$ for elements $a, b \in M$ need not be well-defined, 
since some element $c$ other than $b$ could satisfy $ac = ab$.  However, the expression $u_1 \cdots u_r / u_i$ appearing in Definition~\ref{d:bullet}(b) is 
well defined and equals $u_1 \cdots u_{i-1}u_{i+1} \cdots u_r$ since the monoid $M$ is cancellative.  In general, much of our intuition for factorization theory 
breaks down in a non-cancellative setting, so the cancellative hypothesis in Definition~\ref{d:omega} (as well as for most other factorization invariants) is crucial.  
\end{remark}

A more general form of Proposition~\ref{p:bullet} appeared as \cite[Proposition~3.3]{semitheor}.  

\begin{prop}\label{p:bullet}
Given any commutative, cancellative, atomic monoid $M$, 
$$\omega(x) = \sup\{r : u_1 \cdots u_r \in \bul(x), u_i \emph{\text{ irreducible}}\}$$
for each element $x \in M$.  
\end{prop}

\begin{proof}
First, suppose $\sup\{r : u_1 \cdots u_r \in \bul(x), u_i {\text{ irreducible}}\} = \infty$.  Then there exists bullets 
$u_1 \cdots u_r$, which $x$ divides but no $u_i$ can be omitted, for arbitrarily large $r$.  
By definition, this means $\omega(x) = \infty$.  

Conversely, let $m = \max\{r : u_1 \cdots u_r \in \bul(x), u_i \,\,{\text{irreducible}}\}$.  For any bullet $u_1 \cdots u_r \in \bul(x)$, 
$x$ divides $u_1 \cdots u_r$, but omitting any $u_i$ yields a product which $x$ does not divide, 
so $\omega(x) \ge r$.  This means $\omega(x) \ge m$.  Now fix $r > m$.  Any product 
$u_1 \ldots u_r$ of irreducibles which $x$ divides is not a bullet, so there is some $u_i$ with 
$x \mid u_1 \cdots u_r / u_i$.  This means $\omega(x) < r$, which proves the desired equality.  
\end{proof}

\begin{example}\label{e:bullet}
Let $M = \{6a + 9b + 20c : a,b,c \in \NN\}$ denote the McNugget Monoid introduced in Example~\ref{e:mcnugget} and consider again $6 \in M$.  
As a collection of irreducible elements, any bullet $B = u_1 + \cdots + u_r$ for $6$ is a factorization (of some other element in $M$) such 
that $(u_1 + \cdots + u_r) - 6 \in M$ but removing any $u_i$ from $B$ would be no longer divisible by $6$.  In fact, the only bullets for $6$ 
are the expressions $6$, $9 + 9$ and $20 + 20 + 20$.  Since the latter has maximal length, Proposition~\ref{p:bullet} yields $\omega_M(6) = 3$.

It is worth noting that a bullet is not just an element in $M$, but a specific sum of irreducibles.  Let $B$ denote 
the bullet $20 + 20 + 20$ for $6$.  The expression $6 + \cdots + 6$ consisting of 10 copies of 6 has the same value 
in $M$ as $B$ (namely, 60), but this sum is not a bullet for $6$.  Indeed, omitting any summand (or all but one of 
them, for that matter) yields a sum that $6$ still divides in $M$.  Additionally, we verify that $B$ is a bullet for 
$6$ by checking that $(20 + 20 + 20) - 20 - 6 = 34 \notin M$.  Even though $(20 + 20 + 20) - 9 - 6$ does lie in $M$, 
$9$ is not one of the irreducibles comprising $B$.  

\end{example}

We now give some basic properties of the $\omega$-function, which appeared as \cite[Theorems~2.3 and~2.4]{prime} and \cite[Proposition~2.1]{andalg}.  
Notice that both definitions of $\omega$ are used in the proof of Proposition~\ref{p:basicprops}.  

\begin{prop}\label{p:basicprops}
Fix a commutative, cancellative, atomic monoid $M$.  
\begin{enumerate}

\item 
The set $\{\omega(x) : x \in M\}$ is unbounded.  

\item 
$\omega$ is subadditive, that is, $\omega(xy) \le \omega(x) + \omega(y)$ for all $x, y \in M$.  

\item 
If $p \in M$ is prime, then $\omega(px) = \omega(x) + 1$ for all $x \in M$.  

\end{enumerate}
\end{prop}

\begin{proof}
For any $x \in M$, each factorization of $x$ is a bullet for $x$.  As such, $\omega(x)$ is bounded below by the 
largest length of a factorization of $x$.  Thus, any irreducible $u \in M$ satisfies $\omega(u^k) \ge k$ for $k \ge 1$, 
so $\{\omega(x) : x \in M\}$ is unbounded.  

Now, fix two elements $x, y \in M$, and suppose $xy \mid u_1 \cdots u_r v_1 \cdots v_s$ for $r = \omega(x)$ and $s > \omega(y)$.  
Then for some subsets $T \subset \{1, \ldots, r\}$ and $U \subset \{1, \ldots, s\}$ with $|T| \le \omega(x)$ and $|U| \le \omega(y)$, 
we have $xy \mid (\prod_{i \in T} u_i)(\prod_{j \in U} v_j)$.  This proves the desired inequality.  

Lastly, fix $p \in M$ prime and $x \in M$.  Any bullet $u_1 \cdots u_r$ for $px$ gives $p \mid u_1 \cdots u_r$, 
so $p \mid u_j$ for some $j$, say $j = 1$.  Since $u_1$ is irreducible, $p = au_1$ for some unit $a$.  
Cancellativity of $M$ yields $x \mid u_2 \cdots u_r$, so $u_2 \cdots u_r$ is a bullet for $x$.  
This proves $\omega(px) = \omega(x) + 1$.  
\end{proof}

\subsection{Infinite $\omega$-values}

In general, the $\omega$-primality function can take the value $\infty$ for some elements, as Example~\ref{e:infiniteomega} demonstrates.   
With the exception of this example, all of monoids in this paper attain only finite $\omega$-values.  

\begin{example}[{\cite[Example~4.7]{tame}}]\label{e:infiniteomega}
Fix an irrational number $\alpha > 1$ and let 
$$A = \{(x,y) \in \NN^2 : y \le \alpha x\} \subset \NN^2$$ 
denote the additive submonoid of $\NN^2$ consisting of the points below the line $y = \alpha x$.  
The element $(1,1) \in M$ has infinite $\omega$-value.  The idea of the proof is that, by the irrationality of $\alpha$, the line $y = \alpha x$ 
contains no integer points other than the origin; however, integer points in $A$ can lie arbitrarily close to it.  
In particular, $(1,1)$ divides some multiple of every element of $A$, but for each $n \in \NN$, there exists 
an irreducible $u_n \in A$ close enough to the line $y = \alpha x$ that $nu_n - (1,1) \notin A$ as well.  This 
bounds the $\omega$-value of $(1,1)$ below by $n$ for all $n \in \NN$.  Proposition~\ref{p:basicprops}.2 implies 
that every nonzero element of $A$ also has infinite $\omega$-value.  See \cite[Example~4.7]{tame} for the full argument.  
\end{example}

\subsection{$\omega$-values in Block and Krull monoids}\label{s:block}

Let $G$ be an Abelian group, written additively.  The \emph{free commutative monoid over $G$} is the set of unordered, finite sequences of terms from $G$, where repetition of terms is allowed.  Formally, this is given by $$\mathcal F(G) = \{g_1^{a_1} \cdots g_r^{a_r} \, : \, g_i \in G \,\, \text{distinct}, a_i \in \NN\}.$$ The elements of $\mathcal F(G)$ can be thought of as monomials whose variables are the elements of $G$ and whose multiplication is given by usual monomial multiplication.  As a monoid, $\mathcal F(G)$ only depends on $|G|$; in particular, if $|G| = n$ is finite, then $\mathcal F(G) \cong \NN^n$ under the map that sends $g_i$ to the $i$-th standard basis vector in $\NN^n$.  However, the following submonoid of $\mathcal F(G)$, called the block monoid $\mathcal B(G)$, does depend on $G$ and is used extensively in additive number theory.

\begin{defn}\label{d:blockmonoid}
Fix an Abelian group $G$ with identity $0 \in G$, and let $\mathcal F(G)$ denote the free commutative monoid over $G$.  The \emph{block monoid} $\mathcal B(G)$ is the submonoid $$\mathcal B(G) = \{g_1^{a_1} \cdots g_r^{a_r} \in \mathcal F(G) \, : \, a_1g_1+ \cdots + a_rg_r = 0 \in G\} \subset \mathcal F(G),$$ whose elements (called \emph{blocks} or \emph{zero-sum sequences}) evaluate to the identity upon applying the group operation.  
\end{defn}

For a full treatment of block monoids, see~\cite{block,block_delta,nonuniq}.

\begin{example}\label{e:blockmonoid}
Consider a cyclic, Abelian group $G = \{0, g, -g\} \cong \mathbb Z_3$ of order $3$ and its corresponding block monoid $\mathcal B(G)$.   The block monoid $\mathcal B(G)$ can be written as $$\mathcal B(G) = \{0^{a_0}g^{a_1} (-g)^{a_2} \, : \, a_0\! \cdot\! 0 + a_1\! \cdot\! g + a_2 \!\cdot\! (-g) = 0 \in G \}.$$  The above algebraic condition for membership in $\mathcal B(G)$ can further be reduced to the number-theoretic condition $a_1 \equiv a_2 \bmod 3$.  A quick computation shows that the four irreducible elements (also called \emph{minimal zero-sum sequences}) in this monoid are precisely $0, g^3, (-g)^3,$ and $g(-g)$, where $0$ is the unique prime element.  This monoid has non-unique factorizations as, for example, $g^3(-g)^3$ can be factored as $(g^3)((-g)^3)$ or $(g(-g))^3$.   Computing the $\omega$-values for these irreducible elements, we obtain the following:

\begin{itemize}
\item[$\cdot$] $\omega(0) = 1$ as $0$ is its unique (maximal) bullet of length $1$;

\item[$\cdot$] $\omega(g^3) =  \omega((-g)^3) = 3$ as $(g(-g))^3$ is a maximal bullet of length $3$;

\item[$\cdot$] $\omega(g(-g)) = 2$ as $(g^3)((-g)^3)$ is a maximal bullet of length $2$.
\end{itemize}
\end{example}

It is not a coincidence that, for these irreducible elements in $\mathcal B(G)$, the $\omega$-value coincides with the number of ``copies'' of the group elements in that block.  The following proposition establishes this fact and is proved more generally in~\cite{origin}.

\begin{prop}\label{p:block}
Let $G \not \cong \ZZ_2$ be an Abelian group and $\mathcal B(G)$ its block monoid.  Every irreducible element $u = g_1^{a_1} \cdots g_k^{a_k} \in \mathcal B(G)$ with $n = a_1 + \cdots + a_k$ satisfies $\omega(u) = n$.  
\end{prop}

\begin{proof}
Fix an irreducible element $u = g_1^{a_1} \cdots g_k^{a_k} \in \mathcal B(G)$, with $g_i \ne g_j$ for $i \ne j$, and let $n = a_1 + \cdots + a_k$.  First, suppose $u \mid u_1 \cdots u_m$ for elements $u_1, \ldots, u_m \in \mathcal B(G)$.  Writing $ux = u_1 \cdots u_m$ for some $x \in \mathcal B(G)$, each $g_i$ in $u$ must occur in some $u_j$.  Comparing left and right sides of the equation, we can find a subset $T \subset \{1, \ldots, m\}$ with $|T| \le n$ such that $u \mid \prod_{j \in T} u_j$.  Since $|T| \le n$, this shows that $\omega(u) \le n$.  

Next, we construct a bullet for $u$ of length $n$.  If $u = g(-g)$ for some $g \in G$ with $|g| \ne 2$, then $\left(g_i^{|g_i|}\right)\left((-g_j)^{|g_j|}\right)$ is a maximal bullet for $u$.  If $u = g^2$ with $|g| = 2$, then since $G \not \cong \ZZ_2$, there exists an element $g' \notin \<g_i\> \subset G$, and $(g_i)(g')(-(g_ig'))$ is a maximal bullet for $u$.  For the remaining cases, if $|g_i| = 2$, then $a_i = 1$, and if $|g_i| > 2$, then $-g_i$ does not occur in $u$.  This means $\left(g_1(-g_1)\right)^{a_1} \cdots \left(g_k(-g_k)\right)^{a_k}$ is a bullet for $u$ of length $n$, since each $g_i$ occurs exactly $a_i$ times.  This completes the proof.  
\end{proof}

Block monoids fall into a larger, important category of algebraic objects called \emph{Krull monoids}.  A Krull monoid $M$ is a submonoid of a factorial monoid $D$ (i.e., a monoid where every element has a unique irreducible factorization) such that for each $a, b \in M$, we have $a \mid_M b$ implies $a \mid_D b$.  An analogue of Proposition~\ref{p:block} is true for Krull monoids satisfying certain conditions, where the $\omega$-value of an irreducible element $u \in M$ is the (unique) factorization length of $u$ in $D$.   See~\cite{origin, tame} for more details.

\begin{remark}\label{r:block}
While Proposition~\ref{p:block} easily finds the $\omega$-values for the irreducible elements of block monoids, no known formula exists for $\omega$-values of reducible elements.  Consider, for example, $g^3(-g)^3 \in \mathcal B(G)$, where $G$ is a cyclic, Abelian group of order $3$.  The bullets for this element are $(g^3)((-g)^3)$ and $(g(-g))^3$, with the latter being maximal.  Thus, $\omega(g^3(-g)^3) = 3$ even though its factorization length in the factorial monoid $\mathcal F(G)$ is $6$.  Therefore, Proposition~\ref{p:block} cannot in general be extended to provide a computation of $\omega$-values for reducible elements in block monoids.  See Problem~\ref{p:blockmonoid} for more details.
\end{remark}

\begin{remark}\label{r:irreducibles}
Much of the seminal literature on $\omega$-primality focuses on understanding $\omega$-values for irreducible elements in large classes of cancellative, commutative, atomic monoids.  In such monoids, the presence of non-unique factorizations is equivalent to the existence of irreducible elements which are not prime; thus, measuring ``how far from prime" irreducibles are was a natural course of study.  In fact, understanding $\omega$-values for irreducible elements has provided a deeper understanding of the important concept of tameness of monoids.  See \cite{tame} for a definition of tameness and for details on its connection to $\omega$-primality.
\end{remark}

\section{$\omega$-values in numerical monoids}\label{s:numerical}

The Chicken McNugget monoid discussed in Section~\ref{s:general} is one example of an important class of algebraic objects called numerical monoids.  Much is known about the algebraic structure of numerical monoids, and a deeper understanding of their factorization theory has emerged as a result of studying their $\omega$-values.

\subsection{$\omega$-primality and bullets in numerical monoids}

We begin with basic definitions related to numerical monoids.  See~\cite{numerical} for a full treatment.

\begin{defn}\label{d:numerical_monoid}
Fix relatively prime positive integers $n_1, n_2, \ldots, n_k \in \NN$, and consider the subset of $\NN$ defined by 
$$\langle n_1, n_2, \ldots, n_k \rangle = \{c_1n_1 + c_2n_2 + \cdots + c_kn_k  : c_i \in \mathbb N \}.$$  
Along with the additive structure inherited from $\NN$, this set forms the \emph{numerical monoid generated by $\{n_1, n_2, \ldots, n_k\}$}.  
\end{defn}

\begin{remark}\label{r:generating_set}
The elements $n_1, n_2, \ldots, n_k$ in Definition~\ref{d:numerical_monoid} are referred to as a \emph{generating set} for $\Gamma$.  
Even though different generating sets may yield the same numerical monoid, it can be shown that there exists a \emph{unique} minimal 
generating set (with respect to set-theoretic inclusion) for each numerical monoid.  In this paper, we will always assume that the 
given generating set for a numerical monoid is minimal, since the unique minimal generating set can be obtained from any non-minimal 
generating set by simply omitting the redundant generators. 

 Furthermore, the assumption in Definition~\ref{d:numerical_monoid} that that $\gcd(n_1, n_2, \ldots, n_k) = 1$ assures that the the numerical monoid is co-finite in $\NN$.  In cases where the generators satisfy $\gcd(n_1, n_2, \ldots, n_k) = d > 1$, the resulting monoid is a submonoid of $d\NN$; such a monoid, though, is naturally isomorphic to the co-finite numerical monoid generated by $\left\{n_1/d, n_2/d, \ldots, n_k/d\right\}.$
\end{remark}

The elements in the unique minimal generating set for a numerical monoid are precisely its irreducible elements.  
Since there are only finitely many irreducibles, we are able to represent a factorization of an element by simply counting how many times 
each generator appears.  This motivates Definition~\ref{d:factorization}.  

\begin{defn}\label{d:factorization}
Let $\Gamma$ be a numerical monoid minimally generated by $n_1, n_2, \ldots, n_k$, 
and fix $n \in \Gamma$.  A vector $\vec a = (a_1, a_2, \ldots, a_k) \in \NN^k$ is a 
\emph{factorization} of $n$ if $n = \sum_{i=1}^k a_in_i.$  The \emph{length} of the 
factorization $\vec a$, denoted by $|\vec a|$, is given by $|\vec a| = \sum_{i=1}^k a_i$. 
\end{defn}
 
Much of the recent work on understanding the $\omega$-function on numerical monoids 
takes advantage of these factorization representations, and many of the definitions 
described in Section~\ref{s:general} have more tractable descriptions in the 
numerical monoid setting.  In particular, the definition of $\omega$-primality and bullets
of an element in a numerical monoid have more straightforward descriptions.

\begin{defn}\label{d:omega_numerical}
Let $\Gamma = \langle n_1, n_2, \ldots, n_k \rangle$ be a numerical monoid 
with irreducible elements $n_1, n_2, \ldots, n_k$.  For $n \in \Gamma$, 
define $\omega(n) = m$ if $m$ is the smallest positive integer with the 
property that whenever $\vec a \in \NN^k$ satisfies $\sum_{i=1}^k a_in_i - n \in \Gamma$ 
with $|\vec a| > m$, there exists a $\vec b \in \NN^k$ with $b_i \leq a_i$ for each $i$ 
such that $\sum_{i=1}^k b_i n_i - n \in \Gamma$ and $|\vec b| \leq m$.  
We say that $\vec a \in \NN^k$ is a \emph{bullet} for $n$ if 
$\left(\sum_{i=1}^n a_i n_i\right) - n \in \Gamma$ and 
$(\sum_{i=1}^k a_in_i) - n_j - n \not \in \Gamma$ whenever $a_j > 0$.  
A bullet $\vec a$ for $n$ is \emph{maximal} if $|\vec a| = \omega(n)$.
\end{defn}

\begin{remark}
The authors of~\cite{tame} prove that a certain category of monoids, which includes finitely-generated monoids, attain only finite $\omega$-values.  Thus, $\omega(n) < \infty$ for all $n$ in a numerical monoid.  Therefore, by Proposition~\ref{p:bullet}, a maximal bullet is always guaranteed to exist for every element in a numerical monoid.
\end{remark}

\begin{example}\label{e:3-7_bullets}
Consider the numerical monoid $$\Gamma = \langle 3, 7 \rangle = \{0,3,6,7,9,10,12,13,14,15,\ldots\}.$$
The element $9 \in \Gamma$ has only two bullets: $(3,0)$ and $(0,3)$, both of which are maximal.  Therefore, $\omega(9) = 3$.
\end{example}

In numerical monoids, the set of bullets for a fixed element $n$ lies in a bounded subset of $\NN^k$ that 
can be easily computed.  We describe how to find this bounded set in Lemma~\ref{l:bounded} and give 
Algorithm~\ref{a:numerical} for computing $\omega$-values in any numerical monoid.  Both Lemma~\ref{l:bounded} 
and Algorithm~\ref{a:numerical} appeared in~\cite{andalg}.  

\begin{lemma}\label{l:bounded}
Fix a numerical monoid $\Gamma = \langle n_1, n_2, \ldots, n_k \rangle$ and $n \in \Gamma$.  
\begin{enumerate}
\item%
For every $1 \leq i \leq k$, there exists $b_i > 0$ such that $b_i\vec e_i$ is a bullet for $n$, where $\vec e_i$ is the $i$-th unit vector. 

\item%
Every bullet $\vec a = (a_1,a_2, \ldots, a_k)$ for $n$ satisfies $a_i \leq b_i$ for all $i$.

\end{enumerate}
\end{lemma}

\begin{proof}
Consider the factorization $n\vec e_i$ corresponding to the expression $n\cdot n_i \in \Gamma$.  
Notice that $n \cdot n_1 - n = n(n_1 - 1) \in \Gamma$, so $n$ divides $n \cdot n_1$ in $\Gamma$.  
Choose $b_i$ minimal such that $b_in_i - n \in \Gamma$.  Minimality ensures that 
$b_in_i - n_i - n \notin \Gamma$, so $b_i\vec e_i$ is a bullet for $n$.  

Now, suppose $\vec a$ is a factorization with $a_i > b_i$ and $(\sum_{j=1}^k a_jn_j) - n \in \Gamma$.  
Since $b_in_i - n \in \Gamma$, we can write 
$$(\textstyle\sum_{j=1}^k a_jn_j) - n_i - n = (b_in_i - n) + (a_i - b_i - 1)n_i + \textstyle\sum_{j \ne i} a_jn_j \in \Gamma,$$
so $\vec a$ is not a bullet for $n$.  
\end{proof}

\begin{alg}\label{a:numerical}
For $n \in \Gamma = \<n_1, \ldots, n_k\>$ with each $n_i$ irreducible, compute $\omega(n)$.  












\begin{algorithmic}
\Function{OmegaPrimality}{$\Gamma$, $n$}
\ForAll{$i \le k$}
	\State Find $b_i$ minimal such that $b_in_i - n \in \Gamma$
\EndFor
\State $M \gets 0$ 
\ForAll{$\vec a \in \NN^k$ with $a_i \le b_i$ for all $i \le k$}
	\State $b \gets \sum_{i = 1}^k a_in_i$
	\If{$b - n \in \Gamma$ and $b - n - n_i \notin \Gamma$ whenever $a_i > 0$}
		\State $M \gets \max(M,\sum_{i = 1}^k a_i)$
	\EndIf
\EndFor
\State \Return $M$
\EndFunction
\end{algorithmic}
\end{alg}

\begin{remark}
It is worth noting that Algorithm~\ref{a:numerical} is highly inefficient, as the space searched for 
bullets is much larger than is required.  Significant improvements on the search space are discussed 
in~\cite[Remarks~5.9]{semitheor}, and an upcoming paper gives an algorithm for computing 
$\omega$-primality by inductively computing bullet sets.  Additionally, \cite[Algorithm~3]{compasymp} 
gives an algorithm for computing $\omega$-primality of elements in any finitely-generated monoid.  
\end{remark}

We now give some examples to see Algorithm~\ref{a:numerical} in action.  

\begin{example}\label{e:mcnugget_bullets}
Returning to the McNugget Monoid $M = \langle 6, 9, 20 \rangle$ from Example~\ref{e:mcnugget}, we can compute 
the three component bounds for $n = 35 \in M$.  Proceeding as in Algorithm~\ref{a:numerical}, the factorization 
$(14,0,0)$ has the property that $(14\cdot 6 + 0 \cdot 9 + 0 \cdot 20) - 35 = 49 \in M$, however 
$(14 \cdot 6 + 0 \cdot 9 + 0 \cdot 20) - 35 - 6 = 43 \not \in M$.  Thus, any bullet 
$\vec a = (a_1,a_2,a_3)$ for $35$ will have that $a_1 \leq 14$.  Similarly, $(0,9,0)$ and $(0,0,4)$ 
are bullets for $n = 35$, and therefore $a_2 \leq 9$ and $a_3 \leq 4$.  These component-wise bounds 
give a bounded region of $\NN^3$ in which the remaining bullets may occur.  Automating the search for 
the remaining bullets for $35$ in \texttt{Sage}~\cite{sage} produces the following:  
$$(14,0,0), (11,1,0), (8,3,0), (5,5,0), (2,7,0), (0,9,0), (4,0,1), (1,1,1), (0,3,1), (0,0,4).$$
Here, $(14,0,0)$ has length $14$ and is the unique maximal bullet.  Therefore, $\omega(35) = 14$.  
Notice that the computation of this $\omega$-value would be very difficult without the introduction of bullets.
\end{example}

\begin{example}\label{e:11-13-15_bullets}
Consider the numerical monoid $\Gamma = \<11,13,15\>$ with three irreducible elements and $58 \in \Gamma$.  Our algorithm, 
implemented in \texttt{Sage}, gives us that $(5,6,0)$ is a maximal bullet for $58$ (and thus $\omega(58) = 11$).  
Further investigation shows that for all $j \geq 0$, the element $58 + 11j \in \Gamma$ has $(5+j,6,0)$ as a maximal bullet.  
Therefore, maximal bullets need not have only one non-zero component (c.f., Examples~\ref{e:3-7_bullets} and~\ref{e:mcnugget_bullets}); 
in fact, these types of bullets can occur infinitely many times in a single numerical monoid.  See Problem~\ref{p:maxbullets}.  
\end{example}

\subsection{Numerical monoids with $2$ irreducible elements}

A complete understanding of the $\omega$-function for numerical monoids would include a closed form 
for $\omega(n)$ for any element in a numerical monoid $\Gamma$.  While this is a very lofty goal, 
some progress has been made in more manageable numerical monoids, such as those with only two irreducible 
elements.  Two separate results give different forms for $\omega(n)$ when $n \in \langle n_1, n_2 \rangle$.

\begin{thm}\label{t:2genomega}
Let $1< n_1 <n_2$ be relatively prime natural numbers, and consider the numerical monoid $\Gamma = \langle n_1, n_2 \rangle.$  

\begin{enumerate}[(a)]
\item \emph{(\cite[Theorem~4.4]{andalg})}
Let $(a_1,a_2)$ be a factorization for $n \in \Gamma$.  Then, 
$$\omega(n) = \max \left\{\left\lceil \frac{a_2}{n_1} \right\rceil n_2 + a_1, \left\lceil \frac{a_1}{n_2} \right\rceil n_1 + a_2\right\}.$$

\item \emph{(\cite[Theorem~4.4]{quasi})}
Let $n \in \Gamma$ be sufficiently large, and write $n = qn_1 + r$ with $0 \leq r < n_1$, 
and let $a \geq 0$ be minimal such that $an_1 \equiv r \bmod n_2$.  Then, $\omega(n) = q+a.$

\end{enumerate}
\end{thm}

While both results produce the same $\omega$-values, they have different advantages.  
The form given in Theorem~\ref{t:2genomega}(a) holds for every element in the monoid 
but requires a separate computation for each $n$.  The form in Theorem~\ref{t:2genomega}(b) 
only holds for sufficiently large $n$, but since there are at most $n_1$ possibilities 
for $a$, it is easier to provide a description of the $\omega$ function that more clearly 
demonstrates its growth.

\begin{example}\label{e:2genomega}
The $\omega$-function for the monoid $\Gamma = \langle 3,7 \rangle$ can be written as a piecewise function with 
$3$ cases, corresponding to the three possible values for $a$.  Writing an element 
$n \in \Gamma$ as $n = 3q + r$ with $0 \leq r < 3$, one can easily compute the minimal 
natural number $a$ such that $3a \equiv r \bmod 7$ for $r = 0,1,2$.  Doing so yields 
$a = 0, 5, 3$, respectively.  Thus, we have that 
$$\omega(n) = \left\{\begin{array}{cc}
q + 0 & \text{ for } r = 0 \\
q+5 & \text{ for } r = 1 \\
q+3 & \text{ for } r = 2
\end{array} \right.,$$
which holds for all $n \geq 7$.
\end{example}

\subsection{The asymptotics of $\omega(n)$}

Theorem ~\ref{t:2genomega}(b) also points out an interesting property about the behavior 
of $\omega$-function on a numerical monoid.  Since $q = \frac{n-r}{n_1}$ and $0 \leq r < n_1$, 
we can re-write each $q+a$ expression as a linear function in the variable $n$ with 
slope $\frac{1}{n_1}$.  Thus, the $\omega$-function on numerical monoids with two irreducible 
elements is eventually a collection of up to $n_1$ discrete linear functions.  This type of 
function, called a quasilinear function, is defined below.  See~\cite{ec} for a full treatment of quasipolynomials.

\begin{defn}\label{d:quasi}
A map $f: \NN \to \NN$ is \emph{quasilinear} if $f(n) = a_1(n) n + a_0(n)$, 
where $a_1, a_2: \NN \to \QQ$ are each periodic and $a_1(n)$ is not identically zero.  
\end{defn}

While the closed form for $\omega(n)$ given in Theorem~\ref{t:2genomega}(b) shows that 
$\omega$ is eventually quasilinear on numerical monoids with two irreducible elements, 
this actually holds for every numerical monoid.  Theorem~\ref{t:quasi} appeared as 
\cite[Theorem~3.6]{quasi} and more generally as~\cite[Corollary~19]{compasymp}.  

\begin{thm}\label{t:quasi}
Let $\Gamma = \langle n_1, n_2, \ldots, n_k \rangle$ be a numerical monoid.  The $\omega$-primality function is eventually quasilinear.  
More specifically, there exists an explicit $N_0 > 0$ (computable in terms of the minimal generators of $\Gamma$) such that 
$\omega(n) = \frac{1}{n_1}n + a_0(n)$ for every $n > N_0$, where $a_0(n)$ is periodic with period dividing $n_1$.  
\end{thm} 

Theorem ~\ref{t:quasi} becomes most striking when plotting the function $\omega: \Gamma \to \NN$.  
The eventual quasilinearity manifests graphically as a collection of discrete lines with identical 
slopes of $\frac{1}{n_1}$.  The graphs of the $\omega$-function for the McNugget monoid 
$M = \langle 6, 9, 20 \rangle$ and the $2$-generator monoid $\langle 3,7 \rangle$ are given in Figure 1.  

\begin{figure}\label{f:graph}
\begin{center}
\includegraphics[width=2.5in]{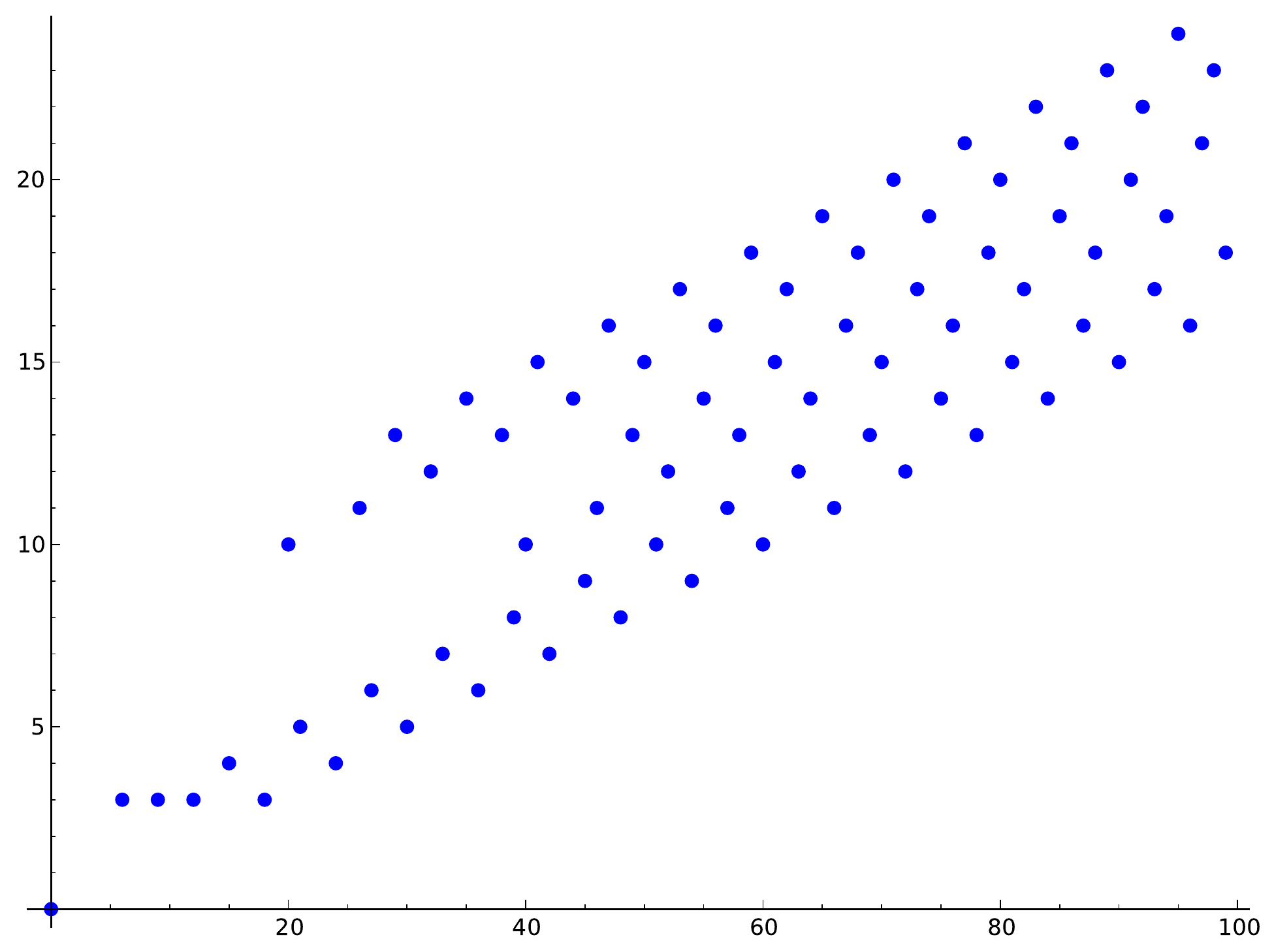}
\hspace{0.5in}
\includegraphics[width=2.5in]{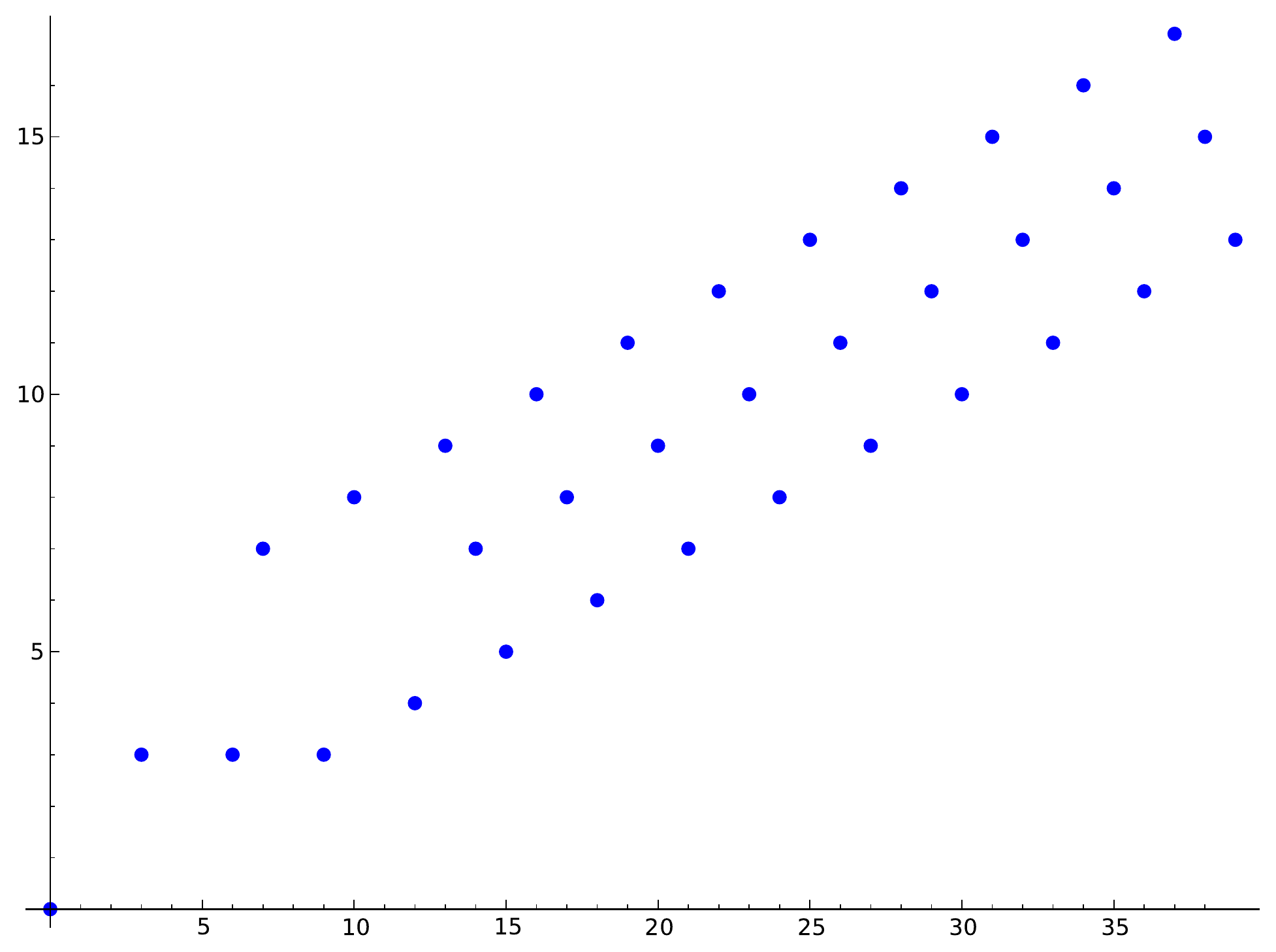}
\caption{The graphs of $\omega(n)$ for the McNugget Monoid $\langle6,9,20 \rangle$ and $\langle3,7 \rangle$.}
\end{center}
\end{figure}

This eventual quasilinearity also can be used to describe the asymptotic growth of $\omega$-values in numerical monoids.  
The following result appears for $2$ generators in ~\cite{andalg} and in full generality in~\cite{compasymp,quasi}.

\begin{cor}[{\cite[Corollary~3.9]{quasi}}]\label{c:asymptotics}
For any numerical monoid $\Gamma = \langle n_1, n_2, \ldots, n_k \rangle$ with $n_1$ the smallest irreducible element, we have that 
$$\lim_{n \to \infty} \frac{\omega(n)}{n} = \frac{1}{n_1}.$$
\end{cor}

\section{Open problems}\label{s:open}

We now give some open problems that have recently arisen in the study of $\omega$-primality.  As previously stated, all problems stated here should be accessible to an advanced undergraduate or a graduate student.  

\subsection{$\omega$-values of generators in numerical monoids}
\label{s:newnumerical}

One natural place to study $\omega$-values in numerical monoids is to compute these values for the irreducible elements 
themselves (i.e, the generators of the monoid).  In numerical monoids $\langle n_1, n_2 \rangle$ with two irreducible 
elements $n_1$ and $n_2$, ~\cite{andalg} shows that the $\omega$-value of a generator is equal to the generator itself.  
In fact, they give the following biconditional statement as a consequence of Theorem~\ref{t:2genomega}(a).  

\begin{cor}\cite[Theorem 4.5]{andalg}\label{c:gen_2gen}
For relatively prime $n_1, n_2 \in \NN$ greater than $1$, a non-identity element $x \in \langle n_1, n_2 \rangle$ 
has $\omega(x) = x$ if and only if $x$ is either $n_1$ or $n_2$.  
\end{cor}

Corollary~\ref{c:gen_2gen} shows that in a numerical monoid with $2$ generators $n_1 < n_2$, the omega values of the 
irreducible elements have the ordering $\omega(n_1) < \omega(n_2)$.  For numerical monoids with $3$ minimal generators 
$n_1 < n_2 < n_3$, the ordering of their $\omega$-values is much more difficult to predict.  For example, the numerical 
monoid $\langle  6,8,13 \rangle$ has $\omega(6) < \omega (8) < \omega(13)$, while $\langle 7, 8, 12 \rangle$ has 
$\omega(7) > \omega(8) = \omega(12)$.  In fact, computation has shown that almost every possible ordering of the set 
$\{\omega(n_1), \omega(n_2), \omega(n_3)\}$ is attained for some generators $n_1, n_2, n_3 \le 100$.  Only three 
orderings fail to appear, motivating the following conjecture, which appears in~\cite{interval}.

\begin{conj}
Fix a numerical monoid $\Gamma = \<n_1, n_2, n_3\>$ minimally generated by $n_1 < n_2 < n_3$.  
The following orderings of $\omega(n_1)$, $\omega(n_2)$ and $\omega(n_3)$ are not possible: 
\begin{enumerate}[(a)]
\item $\omega(n_1) > \omega(n_2) > \omega(n_3)$,
\item $\omega(n_1) = \omega(n_2) > \omega(n_3)$,
\item $\omega(n_3) < \omega(n_1) < \omega(n_2)$.
\end{enumerate}
\end{conj}

Some progress has also been made for certain classes of numerical monoids.  When $\Gamma = \<n, n+1, n+2, \ldots, n + k\>$ is generated 
by an interval of positive integers, an element $x$ lies in $\Gamma$ if and only if $x \bmod n \le \lfloor \frac{x}{n} \rfloor k$.  
The authors of \cite{interval} used this membership criterion to give a formula for the $\omega$-values of the generators of $\Gamma$ when 
$k = 2$.  Their result, stated here as Theorem~\ref{t:3geninterval}, motivates Problem~\ref{p:interval}.  

\begin{thm}[{\cite[Theorem~1.2]{interval}}]\label{t:3geninterval}
Fix an integer $n \ge 2$.  
For $\Gamma = \<2n - 1, 2n, 2n + 1\>$, we have 
$$\omega(2n - 1) = n \text{ and } \omega(2n) = \omega(2n + 1) = n + 1,$$
and for $\Gamma = \<2n, 2n + 1, 2n + 2\>$, we have 
$$\omega(2n) = n, \omega(2n+1) = n + 2, \text{ and } \omega(2n+2) = n + 1.$$
\end{thm}

This leads to the following open problem.

\begin{prob}\label{p:interval}
Fix a numerical monoid $\Gamma = \<n, n + 1, \ldots, n + k\>$.  
Find a formula for the $\omega$-values of the generators of $\Gamma$.  
\end{prob}

A similar membership criterion exists for numerical monoids $\Gamma = \<n, n+s, \ldots, n+ks\>$ generated by an arithmetic sequence 
of length $k$ with step size $s$, and this criterion specializes to the criterion given above for monoids generated by an interval 
when $s = 1$ (see \cite{omidali} and~\cite{omidalirahmati} for more details).  Thus, Problem~\ref{p:interval} is a special case of 
Problem~\ref{p:arithseq}.  

\begin{prob}\label{p:arithseq}
Fix positive integers $n$, $s$ and $k$, and let $\Gamma$ denote the numerical semigroup minimally generated by $\{n, n+s, \ldots, n+ks\}$.  
Find a formula for the $\omega$-values of the generators of $\Gamma$.  
\end{prob}

Problems~\ref{p:interval} and~\ref{p:arithseq} are potentially solvable because the membership criteria for these classes of numerical monoids facilitate the computation of bullets for the generators.  While it remains a lofty goal to ask for a closed form for the $\omega$-function for general numerical monoids, the membership criteria may be helpful enough to make a closed form for those generated by intervals or arithmetic sequences reasonable.   Furthermore, since every pair of natural numbers forms an arithmetic sequence, the closed form for $\omega(n)$ for $\<n_1, n_2\>$ described in Theorem~\ref{t:2genomega}(a) may lend motivation to this end.  Thus, we mention this extended version of Problem~\ref{p:arithseq}.

\begin{prob}\label{p:closed form}
Find a closed form for the $\omega$-function on a numerical monoid generated by an interval or by an arithmetic sequence.
\end{prob}

\subsection{Eventual quasilinearity}
\label{s:newquasi}

Theorem~\ref{t:quasi} guarantees that the $\omega$-function eventually has the form $\omega(n) = \frac{1}{n_1}n + a_0(n)$ 
for a periodic function $a_0(n)$.  The (minimal) period of $a_0(n)$ is only guaranteed to divide $n_1$, but in all 
existing computational data, the period is exactly $n_1$.  This motivates the following.  

\begin{conj}\label{c:period}
Fix a numerical monoid $\Gamma$ with least generator $n_1$.  
For $n$ sufficiently large, $\omega(n)$ is quasilinear with period exactly $n_1$.  
\end{conj}

The smallest value of $N_0$ in Theorem~\ref{t:quasi} is known as the \emph{dissonance point} of the $\omega$-function.  
The value of $N_0$ in the proof of Theorem~\ref{t:quasi} in \cite{quasi}, while constructive, is far from optimal.  
A more precise bound on the dissonance point would make the quasilinear formula for large $\omega$-values more computationally practical.  

\begin{prob}\label{p:dissonance}
Fix a numerical monoid $\Gamma$, and let $f$ denote the quasilinear function such that 
$f(n) = \omega(n)$ for $n \gg 0$.  Characterize the maximal $n$ such that $\omega(n) \ne f(n)$.  
\end{prob}

Example~\ref{e:3-7_bullets} shows that maximal bullets are not necessarily unique in numerical monoids.  
Understanding which bullets of an element are maximal would provide tremendous insight into more easily 
computing $\omega$-values in numerical monoids.  In~\cite{quasi}, the authors prove that when a numerical monoid has two irreducible elements, the maximal bullets of its elements are of the form $(a_1,0)$ or $(0,a_2)$, In general, though, not all maximal bullets need be supported in only one coordinate.  As Example~\ref{e:11-13-15_bullets} shows, this does not hold true with three or more irreducible elements, motivating the following problem.

\begin{prob}\label{p:maxbullets}
Characterize the maximal bullets for numerical monoids $\Gamma$ with
more than two irreducible elements.  
\end{prob}

\subsection{Algorithms}
\label{s:newalgs}

\excise{

Algorithm~\ref{a:numerical} was essential in discovering many of the results given in Section~\ref{s:numerical}, as it allows for the generation of $\omega$-values sufficient for making and testing conjectures in numerical monoids.  Progress for other classes of monoids would significantly benefit from a more generally applicable algorithm.  One reasonable class of monoids to focus on is finitely-generated monoids.  In~\cite{tame}, Geroldinger and Hassler prove that a certain category of monoids (which includes finitely-generated monoids) attains only finite $\omega$-values.  Furthermore, the finite number of generators would make the encoding of factorizations and bullets comparable to numerical monoids.

\begin{prob}\label{p:generalalgorithm}
Find an algorithm to compute $\omega$-primality for finitely-generated monoids.  
\end{prob}

}

The bullet set of an element depends heavily on the bullet sets of its divisors.  
This observation first appeared in the context of \emph{cover maps} (\cite[Definition~3.4]{quasi}) and was 
crucial in the proof of Theorem~\ref{t:quasi}.  Finding a way to construct the bullet set of an element from 
the bullet sets of its divisors would allow for a ``dynamic programming'' algorithm, which first computes the 
bullet sets for the irreducibles and then walks up the divisibility tree, storing new bullet sets as it goes.  

A dynamic programming algorithm for computing bullet sets would be useful even for numerical monoids.  
It would likely be much faster than the current algorithm by effectively computing many $\omega$-values 
at once, rather than one at a time.  Such an algorithm would be ideal for investigating asymptotic problems 
discussed in Section~\ref{s:newquasi}.  

\begin{prob}\label{p:dynamicalgorithm}
Find an algorithm for computing the bullet set of a monoid element using as input the bullet sets of its divisors.  
\end{prob}

\subsection{Other classes of monoids}
\label{s:newmonoids}

With all of the recent progress on $\omega$-primality in numerical monoids, it is natural to ask 
what can be shown about the $\omega$-function for other classes of monoids.  We now introduce two classes of monoids on which many other factorization invariants have been studied, but little is known about 
the $\omega$-function.  

\begin{defn}\label{d:acm}
Fix positive integers $a, b \in \NN^*$ with $a \leq b$, and suppose $a^2 \equiv a \bmod b$.  
The multiplicative submonoid $M_{a,b} = \{n \in \NN^* : n \equiv a \bmod b\}$ 
of $\NN^*$ is called an \emph{arithmetical congruence monoid}.  
\end{defn}

Notice that the condition $a^2 \equiv a \bmod b$ in Definition~\ref{d:acm} ensures that $M_{a,b}$ 
is closed under multiplication.  The Hilbert monoid in Example~\ref{e:hilbert} is the arithmetical 
congruence monoid $M_{1,4}$.  The factorization theory of arithmetical congruence monoids, as we saw 
for the Hilbert monoid, is heavily determined by prime factorization in $\NN$.  Many factorization invariants 
of arithmetical congruence monoids have been studied (see ~\cite{delta_acm, elas_local, acm_arithmetic, acm_elas}), raising the following open problem.  

\begin{prob}
Study $\omega$-primality on arithmetical congruence monoids.  
\end{prob}

In Section~\ref{s:block}, block monoids $\mathcal B(G)$ over Abelian groups $G$ were introduced, and explicit formulae for the $\omega$-values of their irreducible elements were given.  Computations in Remark~\ref{r:block}, though, show that $\omega$-values of reducible elements in $\mathcal B(G)$ are less predictable, motivating the following problem.

\begin{prob}\label{p:blockmonoid}
Study $\omega$-primality on block monoids.  
\end{prob}

%
%
%

Lastly, we will introduce a relatively new monoid of interest.  Leamer monoids arose while studying the Huneke-Wiegand conjecture, an open problem in commutative algebra related to torsion submodules (see~\cite{HWintersection} and~\cite{leamer} for details).  

Given a numerical monoid $\Gamma$ and a positive integer $s \not \in \Gamma$, the elements of a Leamer monoid are finite arithmetic sequences of step size $s$ that are contained in $\Gamma$.  If the arithmetic sequences $\{n, n+s, \ldots, n+ks\}$ and $\{m, m+s, \ldots, m + \ell s\}$ are contained in $\Gamma$, then so is their set-wise sum $\{n+m, n+m+s, \ldots, n+m + (k + \ell)s\}$; therefore, this operation is closed.  Notice that if we represent the arithmetic sequence $\{n, n+s, \ldots, n+ks\}$ by $(n,k) \in \NN^2$, then the above set-wise sum corresponds to usual component-wise addition in $\NN^2$.  This motivates the following definition.

\begin{defn}\label{d:leamermonoid}
Fix a numerical monoid $\Gamma \subset \NN$ and a positive integer $s \in \NN \setminus \Gamma$ outside of $\Gamma$.  
The \emph{Leamer monoid} is the submonoid 
$$S_\Gamma^s = \{(n,k) \in \NN^2 : n, n + s, \ldots, n + ks \in \Gamma\} \subset \NN^2$$
of $\NN^2$ under component-wise addition.  
\end{defn}

Leamer monoids have been shown to have a factorization theory that is much more complicated than the numerical monoids from which they are derived~\cite{teambob}.  Furthermore, as they are additive submonoids of $\NN^2$, Leamer monoids can be represented graphically.  Figure~\ref{f:leamer} shows the Leamer monoid $S_\Gamma^4$ when $\Gamma = \<13,17,22,40\>$, where the irreducible elements are indicated by larger, darker dots.  As in this example, every Leamer monoid has infinitely many irreducible elements.  However, all irreducible elements are either of the form $(k,1)$ or $(n,m)$, where $n$ is bounded above.  Since these irreducible elements are rather constrained, studying the factorization theory of Leamer monoids becomes more tractable.  

\begin{figure}
\begin{center}
\includegraphics[width=4.5in]{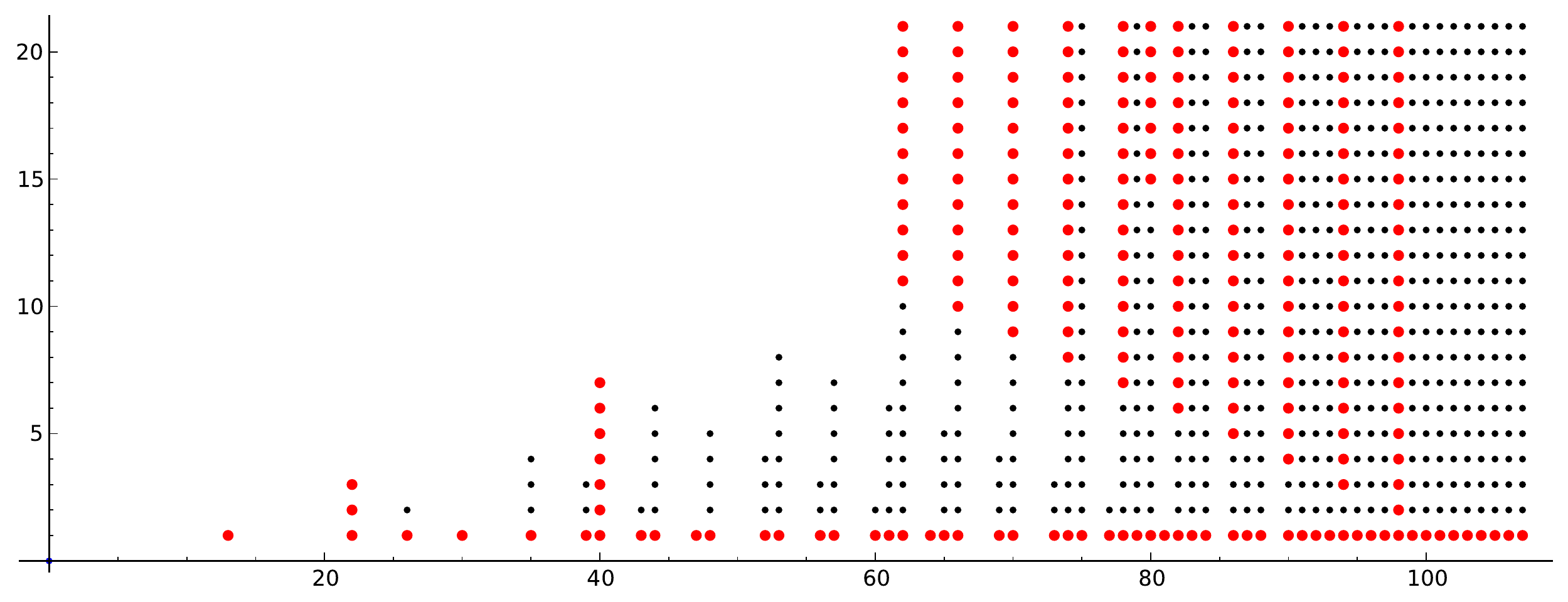}
\caption{The Leamer monoid $S_\Gamma^4$ where $\Gamma = \<13,17,22,40\>$}
\label{f:leamer}
\end{center}
\end{figure}

While numerous factorization invariants, such as Delta sets, elasticity, and catenary degrees, have been studied in Leamer monoids~\cite{teambob}, the $\omega$-value has yet to be analyzed.

\begin{prob}
Study $\omega$-primality on Leamer monoids.  
\end{prob}

\section{Acknowledgements}\label{s:ack}

Much of this work and the open questions included within were generated during the Pacific Undergraduate Research Experience (PURE Math), 
which is funded by National Science Foundation (NSF) grants DMS-1045147 and DMS-1045082, and a supplementary grant from the National Security Agency. The first author was also funded by Dr. Ezra Miller's NSF Grant DMS-1001437.  The authors would like to thank Scott Chapman, Pedro Garc\'ia-S\'anchez, Brian Wissman, and Thomas Barron for helpful comments and discussions.


\noindent {\sc Roberto Pelayo} received his B.A. in Mathematics and Psychology in 2002 from Occidental College and his Ph.D. in Mathematics in 2007 from the California Institute of Technology.  He is currently an Assistant Professor at the University of Hawai`i at Hilo, and his current research interests are low-dimensional topology and the factorization theory of commutative monoids.  

\end{document}